\documentclass[12pt]{article}
\usepackage[left=2.0cm,right=2.0cm,top=2.0cm,bottom=2.0cm]{geometry}
\usepackage{amsfonts}
\usepackage{titlesec}
\usepackage{cite,color,xcolor}
\usepackage{amsmath}
\usepackage{amsthm}
\usepackage{amssymb}
\usepackage{times}
\usepackage{bm}
\usepackage[colorlinks,citecolor=blue,urlcolor=blue]{hyperref}
\usepackage[integrals]{wasysym}
\usepackage[utf8]{inputenc}
\usepackage{mathtools}
\usepackage{setspace}
\usepackage[numbers]{natbib}
\usepackage{cases}
\usepackage{fancyhdr}

\newtheorem{theorem}{Theorem}[section]

\newtheorem{lemma}{Lemma}[section]
\newtheorem{proposition}{Proposition}[section]

\newtheorem{definition}{Definition}[section]
\newtheorem{remark}{Remark}[section]

\usepackage{txfonts}

\newcommand{\bal}{\begin{align}}
\newcommand{\bbal}{\begin{align*}}
\newcommand{\beq}{\begin{equation}}
\newcommand{\eeq}{\end{equation}}
\newcommand{\bca}{\begin{cases}}
\newcommand{\eca}{\end{cases}}

\newcommand{\pa}{\partial}
\newcommand{\fr}{\frac}
\newcommand{\na}{\nabla}
\newcommand{\De}{\Delta}

\newcommand{\cd}{\cdot}

\newcommand{\dd}{\mathrm{d}}

\newcommand{\R}{\mathbb{R}}

\newcommand{\les}{\lesssim}

\newcommand{\F}{\dot{F}}

\begin{document}
\bibliographystyle{plain}
\title{Well-posedness and no-uniform dependence for the Euler-Poincar\'{e} equations in
Triebel-Lizorkin spaces}

\author{Yuanhua Zhong$^{1}$,  Jianzhong Lu$^{2,}$\footnote{E-mail: 3127781939@qq.com; louisecan@163.com(Corresponding author); limin@jxufe.edu.cn; lijinlu@gnnu.edu.cn
},  Min Li$^{3}$, Jinlu Li$^{1}$\\
\small $^1$ School of Mathematics and Computer Sciences, Gannan Normal University, Ganzhou 341000, China\\
\small $^2$  School of Mathematics and Information, Xiangnan University, Chenzhou, 423000, China
\\
\small $^3$ Department of Mathematics, Jiangxi University of Finance and Economics, Nanchang 330032, China
}

\date{\today}

\maketitle\noindent{\hrulefill}

{\bf Abstract:} In this paper, we study the Cauchy problem of the Euler-Poincar\'{e} equations in $\R^d$ with initial data belonging to the Triebel-Lizorkin spaces. We prove the local-in-time unique existence of solutions to the Euler-Poincar\'{e} equations in $F^s_{p,r}(\R^d)$. Furthermore, we obtain that the data-to-solution of this equation is continuous but not uniformly continuous in these spaces.


\vskip0mm\noindent{\hrulefill}

\thispagestyle{empty}
\section{Introduction}
\quad

In this paper, we consider the following Cauchy problem of Euler-Poincar\'{e} equations:
\begin{equation}\label{E-P}\tag{EP}
\begin{cases}
\partial_tm+u\cdot \nabla m+\nabla u^T\cd m+(\mathrm{div} u)m=0, \qquad (t,x)\in \R^+\times \R^d,\\
m=(1-\De)u,\qquad (t,x)\in \R^+\times \R^d,\\
u(0,x)=u_0,\qquad x\in \R^d,
\end{cases}
\end{equation}
where $u=(u_1,u_2,\cdots,u_d)$ denotes the velocity of the fluid, $m=(m_1,m_2,\cdots,m_d)$ represents the momentum. Equivalently, we can write \eqref{E-P} in components,
$$\partial_tm_i+\sum^d_{j=1}u_j\partial_{x_j}m_i+\sum^d_{j=1}(\partial_{x_i}u_j)m_j+m_i\sum^d_{j=1}
\partial_{x_j}u_j=0,~~~~~~~~i=1,2,\cdots,d.$$
According to Yan-Yin \cite{yy}, we can transform the first equation of \eqref{E-P} into the following
nonlocal transport form:
\begin{align}\label{E-P1}
\partial_tu+u\cdot \nabla u= Q(u,u)+R(u,u):=P(u,u),
\end{align}
where
\bal\label{n-Puv}
\begin{cases}
Q(u,v)=-(1-\De)^{-1}\mathrm{div}\Big(\nabla u\nabla v+\nabla u(\nabla v)^T-(\nabla u)^T\nabla v
 -(\mathrm{div} u)\nabla v+\frac12\mathbf{I}(\nabla u:\nabla v)\Big),
\\R(u,v)=-(1-\De)^{-1}\Big((\mathrm{div} u)v+ (\na u)^T\cd v\Big),
\end{cases}
\end{align}
with
\bbal
&(\na u^T)_{i,j}=\pa_{x_i}u_j, \quad (u\cd \na v)_i=\sum^d_{k=1}u_k\pa_{x_k}u_i, \quad (\na u\na v)_{ij}=\sum^d_{k=1}\pa_{x_i}u_k\pa_{x_k}v_j,
\\&\na u:\na v=\sum^d_{i,j=1}\pa_{x_i}u_j\pa_{x_i}v_j, \quad \big((\na u)^T\cd v\big)_i=\sum^d_{j=1}\pa_{x_i}u_jv_j.
\end{align*}

The system \eqref{E-P} was first introduced by Holm et al. in \cite{hmr1,hmr2} as a higher dimensional generalization of the following Camassa-Holm equation (CH):
\begin{align}\label{C-H}\tag{CH}
m_{t}+um_{x}+2u_{x}m=0, \ m=u-u_{xx}.
\end{align}
Indeed, when $d=1$, the Euler-Poincar\'{e} equations reduce to \eqref{C-H}.
The celebrated Camassa-Holm equation \eqref{C-H} was originally discovered by Camassa and Holm \cite{ch} as a bi-Hamiltonian system for modeling and analyzing the nonlinear shallow water waves. It is worth mentioning that the CH equation has peakon solutions of the form $Ce^{-|x-Ct|}\,(C>0)$, which aroused a lot of interest in physics; see for instance \cite{c5,t}. Nowadays, the research on the existence of solutions to the CH equation, their analytical and geometric properties, as well as the theory of well-posedness, has yielded fruitful results. The existence of global weak solutions and dissipative solutions was investigated in \cite{bc1,bc2,xz1} and the references therein. Blow-up phenomena and global existence of strong solutions were discussed in \cite{c2,ce2,ce3,ce4}, while the non-uniform continuity of the CH equation was considered in \cite{hk,hkm,lyz1,lwyz}. Moreover, in \cite{ce2,d2,glmy},
the local well-posedness and ill-posedness for the Cauchy problem of the CH equation were studied. 

The first rigorous analysis work of the Cauchy problem for the Euler-Poincar\'{e} equations \eqref{E-P} is attributed to Chae and Liu \cite{cli}. They eatablished the local existence of weak solutions in $W^{2,p}(\R^d)$ with $p>d$, and proved the local existence of unique classical solutions in $H^{s}(\R^d)$ with $s>d/2+3$. Soon after, Yan and Yin \cite{yy} further discussed the local existence and uniqueness of the solutions to \eqref{E-P} in Besov spaces. On the other hand, Li et al. \cite{lyzz} proved that the solutions to \eqref{E-P} with a large class of smooth initial data blow up in finite time or exist globally in time, which settled an open problem raised by Chae and Liu \cite{cli}. Recently, Luo and Yin \cite{luoy} obtained a new blow-up result in the periodic case by using the rotational invariant properties of the equation. For more results about Euler-Poincar\'{e} equations, we refer to \cite{hs,zyl}.

In recent years, as part of the well-posedness, the continuity properties of the data-to-solution map for Camassa-Holm type equations has received increasing interest from scholars. The related research was initiated by Li-Yu-Zhu in \cite{lyz1}. Later, Li et al. \cite{ldz} proved that for the initial data in $H^{s}(\R^d)$ with $s>1+d/2$, the corresponding solution to \eqref{E-P} is not uniformly continuous dependent. This non-uniformly continuous result was extended to Besov space $B^s_{p,r}(\R^d),s>\max\{1+d/2,3/2\}$ in \cite{ldl}. Very recently, based on Galilean boost method and the symmetric structure, Li and Liu proved in \cite{ll} that the data-to-solution map of the Euler-Poincar\'{e} equations \eqref{E-P} is not uniformly continuous on any open subset $U\subset B^s_{p,r}(\R^d)$ with $s>\max\{1+d/2,3/2\}$.

In the present paper, inspired by \cite{lyz-novi}, we consider the Cauchy problem \eqref{E-P} in Triebel-Lizorkin spaces $F^s_{p,r}(\mathbb{R}^d)$, and finally deduce the well-posedness and no-uniform dependence. To state our main result precisely, we first recall the notion of well-posedness in the sense of Hadamard:
\begin{definition}\label{hloc}
  (Local Well-posedness) The Cauchy problem (\ref{E-P}) is said to be locally well-posed in a Banach space $X$ if the following three conditions hold
\begin{enumerate}
\item (Local existence) For any initial data $u_0\in X$, there exists a short time $T=T(u_0)$ and a solution $u(t)\in\mathcal{C}([0,T),X)$ to the Cauchy problem (\ref{E-P});
\item (Uniqueness) This solution $u(t)$ is unique in the space $\mathcal{C}([0,T),X)$;
\item (Continuous Dependence) The data-to-solution map $u_0\mapsto\mathbf{S}_t(u_0)$ is continuous in the following sense: for any $T_1< T$ and $\epsilon>0$, there exists $\delta >0$, such that if $\|u_0-\widetilde{u}_0\|_X<\delta$, then $\mathbf{S}_t(\widetilde{u}_0)$ exists up to $T_1$ and
$$\|\mathbf{S}_t(u_0)-\mathbf{S}_t(\widetilde{u}_0)\|_{\mathcal{C}([0,T_1],X)}<\epsilon.$$
\end{enumerate}
\end{definition}

Now we can present our main result in this paper.
\begin{theorem}\label{the1}
Let $d\geq 1$ and $1< p, r<\infty$. Assume that $u_0\in F^{s}_{p,r}(\R^d)$ with $s>\max\{\frac32,1+\frac dp\}$. Then there exists some time $T>0$ such that
\begin{enumerate}
  \item system \eqref{E-P1} has a solution $u\in E^{s}_{p,r}(T):=\mathcal{C}([0,T]; F^s_{p,r})\cap \mathcal{C}^1{([0,T]; F^{s-1}_{p,r})}$;
  \item the solutions of \eqref{E-P1} are unique;
  \item the data-to-solution map $u_0 \mapsto u(t)$ is continuous from any bounded
subset of $u_0\in F^s_{p,r}$ into $\mathcal{C}([0,T],F^s_{p,r})$;
\item the data-to-solution map $u_0 \mapsto u(t)$ is not uniformly continuous from any bounded subset of $u_0\in F^s_{p,r}$ into $\mathcal{C}([0,T],F^s_{p,r})$.
\end{enumerate}
\end{theorem}

\begin{remark}
In \cite{yang2015}, the existence and uniqueness of the CH equation have been studied well, while the continuous and non-uniform continuous dependence with respect to initial data seem empty. Note that when $d=1$, the Euler-Poincar\'{e} equations are just the Camassa-Holm equations. Therefore, our result in Theorem \ref{the1} extend and improve the corresponding result in \cite{yang2015}.
\end{remark}

\section{Preliminaries}

In this section, we provide some basic facts about Littlewood-Paley decomposition, nonhomogeneous Besov spaces, nonhomogeneous Triebel-Lizorkin spaces, and their useful properties. Also, we recall  the transport equation theory on Besov spaces and Triebel-Lizorkin spaces, which will be used in the sequel. For more details, the readers may consult monographs \cite{bcd,Triebel}.

\begin{proposition}[\cite{bcd}]
There exists a couple of smooth functions $(\chi,\varphi)$ valued in $[0,1]$, such that $\chi$ is supported in the ball $\mathcal{B}\triangleq \{\xi\in\mathbb{R}^d:|\xi|\leq \frac 4 3\}$, and $\varphi$ is supported in the ring $\mathcal{C}\triangleq \{\xi\in\mathbb{R}^d:\frac 3 4\leq|\xi|\leq \frac 8 3\}$. Moreover,
$$\forall\,\ \xi\in\mathbb{R}^d,\,\ \chi(\xi)+{\sum\limits_{q\geq0}\varphi(2^{-q}\xi)}=1,$$
$$\textrm{supp}\,\ \varphi(2^{-q}\cdot)\cap \textrm{supp}\,\ \varphi(2^{-p}\cdot)=\emptyset,\,\ if\,\ |q-p|\geq 2,$$
$$\textrm{supp}\,\ \chi(\cdot)\cap \textrm{supp}\,\ \varphi(2^{-q}\cdot)=\emptyset,\,\ if\,\ q\geq 1.$$
Then for all $u \in \mathcal{S}'$, we can define the nonhomogeneous dyadic blocks as follows. Let
$$\Delta_q{u}\triangleq 0,\,\ if\,\ q\leq -2,\quad
\Delta_{-1}{u}\triangleq \chi(D)u=\mathcal{F}^{-1}(\chi \mathcal{F}u),$$
$$\Delta_q{u}\triangleq \varphi(2^{-q}D)u=\mathcal{F}^{-1}(\varphi(2^{-q}\cdot)\mathcal{F}u),\,\ if \,\ q\geq 0.$$
Hence, $ u={\sum\limits_{q\in\mathbb{Z}}}\Delta_q{u}$ in $\mathcal{S}'(\mathbb{R}^d)$ is called the nonhomogeneous Littlewood-Paley decomposition of $u$.
\end{proposition}

\begin{definition}[\cite{bcd}]\label{de-b}
Let $s>0, 1< p,q\leq\infty$. The nonhomogeneous Besov space $B^s_{p,q}(\mathbb{R}^d)$ is defined by
$$B^s_{p,q}(\mathbb{R}^d)\triangleq \{f \in \mathcal{S}'(\mathbb{R}^d):||f||_{B^s_{p,q}(\mathbb{R}^d)}\triangleq \big|\big|(2^{qs}||\Delta_q{f}||_{L^p(\mathbb{R}^d)})_{q\geq-1}\big|\big|_{\ell^q}< \infty\},$$
and $ B^\infty_{p,q}(\mathbb{R}^d)\triangleq \cap_{s\in\mathbb{R}}B^s_{p,q}(\mathbb{R}^d).$
\end{definition}

\begin{definition}[\cite{Triebel}]\label{de-t}
Let $s>0, 1<p,r<\infty$. The nonhomogeneous Triebel-Lizorkin space $F^s_{p,r}(\mathbb{R}^d)$ is defined by
$$F^s_{p,r}(\mathbb{R}^d)\triangleq \{f \in \mathcal{S}'(\mathbb{R}^d): ||f||_{L^p(\R^d)}+||f||_{\F^s_{p,r}(\mathbb{R}^d)}< \infty\},$$
with
$$||f||_{\F^s_{p,r}(\mathbb{R}^d)}\triangleq \left\|\left(\sum_{j \in \mathbb{Z}} 2^{j s r}\left|\Delta_j f\right|^r\right)^{\frac{1}{r}}\right\|_{L^p(\R^d)}.$$
\end{definition}

\begin{remark}\label{re-bt}
Let $s,t>0, 1<p,r<\infty$ and $1< q\leq \infty$. It should be emphasized that the following embedding will be often used implicity:
\begin{align*}
&B^s_{p,q}(\R^d)\hookrightarrow B^t_{p,q'}(\R^d)\quad\text{for}\;s>t\quad\text{or}\quad s=t,1< q\leq q'\leq\infty,
\\&F^s_{p,r}(\R^d)\hookrightarrow F^t_{p,r'}(\R^d)\quad\text{for}\;s>t\quad\text{or}\quad s=t,1< r\leq r'<\infty,
\\&F^s_{p,r}(\R^d)\hookrightarrow B^s_{p,\infty}(\R^d).
\end{align*}
\end{remark}

\begin{lemma}[\cite{bcd,Triebel}]\label{em-infi}
Let $s>\frac1p$ and $(p,r,q) \in(1, \infty) \times(1, \infty) \times (1, \infty]$. Then, there exists a constant $C$ such that
\bbal
||f||_{L^\infty(\R^d)}\leq C||f||_{B^s_{p,q}}, \qquad ||f||_{L^\infty(\R^d)}\leq C||f||_{B^s_{p,r}}.
\end{align*}
\end{lemma}

\begin{lemma}[\cite{bcd,Triebel}]\label{le-op}
Let $s>0$ and $(p,r,q) \in(1, \infty) \times(1, \infty) \times (1, \infty]$. There exists a positive constant $C$ such that
\bbal
&||(1-\De)^{-1}f||_{F^s_{p,r}(\mathbb{R}^d)}\leq ||f||_{F^{s-2}_{p,r}(\mathbb{R}^d)}, \qquad ||\na f||_{F^s_{p,r}(\mathbb{R}^d)}\leq ||f||_{F^{s+1}_{p,r}(\mathbb{R}^d)},
\\&||(1-\De)^{-1}f||_{B^s_{p,q}(\mathbb{R}^d)}\leq ||f||_{B^{s-2}_{p,q}(\mathbb{R}^d)}, \qquad ||\na f||_{B^s_{p,q}(\mathbb{R}^d)}\leq ||f||_{B^{s+1}_{p,q}(\mathbb{R}^d)}.
\end{align*}
\end{lemma}

\begin{lemma}[\cite{bcd,chae}]\label{le-pro-es1}
Let $s>0$ and $(p,r,q) \in(1, \infty) \times(1, \infty) \times (1, \infty]$. There exists a positive constant $C$ such that
\bbal
&\|f g\|_{{B}_{p, q}^s(\mathbb{R}^d)} \leq C\left(\|f\|_{L^{\infty}(\mathbb{R}^d)}\|g\|_{{F}_{p, q}^s(\mathbb{R}^d)}+\|g\|_{L^{\infty}(\mathbb{R}^d)}\|f\|_{{F}_{p, q}^s(\mathbb{R}^d)}\right).
\\&\|f g\|_{{F}_{p, r}^s(\mathbb{R}^d)} \leq C\left(\|f\|_{L^{\infty}(\mathbb{R}^d)}\|g\|_{{F}_{p, r}^s(\mathbb{R}^d)}+\|g\|_{L^{\infty}(\mathbb{R}^d)}\|f\|_{{F}_{p, r}^s(\mathbb{R}^d)}\right).
\end{align*}
\end{lemma}

\begin{lemma}[\cite{bcd}]\label{le-pro-es2}
Let $(p,q)\in (1, \infty) \times (1, \infty]$ and $s>\max\big\{1+\frac dp,\frac32\big\}$. Then we have
\bbal
&\|fg\|_{B^{s-2}_{p,q}(\mathbb{R}^d)}\leq C\|g\|_{B^{s-2}_{p,q}(\mathbb{R}^d)}\|g\|_{B^{s-1}_{p,q}(\mathbb{R}^d)}.
\end{align*}
\end{lemma}

\begin{lemma}[\cite{bcd}]\label{le-tr-b}
Let $s>0$ and  $(p,q)\in (1, \infty) \times (1, \infty]$. Assume that $f_0\in B^s_{p,q}$, $g\in L^1(0,T; B^s_{p,q})$. If $f\in L^\infty(0,T; B^s_{p,q})\bigcap \mathcal{C}([0,T]; \mathcal{S}^{'})$ solves the following linear transport equation:
\begin{equation}\label{d-tr}
\quad \partial_t f+v\cdot \nabla f=g,\quad  f|_{t=0} =f_0,
\end{equation}
then there exists a constant $C=C(d,p,r,s)$ such that the following statement holds:
\begin{align}\label{ES2}
\sup_{\tau\in [0,t]}\|f(\tau)\|_{B^{s}_{p,q}}\leq Ce^{CV(t)}\Big(\|f_0\|_{B^s_{p,q}}
+\int^t_0\|g(\tau)\|_{B^{s}_{p,q}}\dd \tau\Big),
\end{align}
with
\begin{align*}
V(t)=
\begin{cases}
\int_0^t \|\nabla v(\tau)\|_{B^{\frac{d}{p}}_{p,\infty}\cap L^\infty}\dd \tau,&\quad\mathrm{if} \; s<1+\frac{d}{p},\\
\int_0^t \|\nabla v(\tau)\|_{B^{s}_{p,q}}\dd \tau,&\quad\mathrm{if} \; s=1+\frac{d}{p},\\
\int_0^t \|\nabla v(\tau)\|_{B^{s-1}_{p,q}}\dd \tau, &\quad \mathrm{if} \;s>1+\frac{d}{p}.
\end{cases}
\end{align*}
If $f=v$, the estimate \eqref{ES2} holds with
\[V(t)=\int_0^t \|\nabla v(s)\|_{L^\infty}\dd s.\]
\end{lemma}

\begin{lemma}[\cite{lyz-novi}]\label{le-tr-t}
Let $s>0$ and  $(p,r)\in (1, \infty) \times (1, \infty)$.  Assume that $f_0\in F^s_{p,r}(\mathbb{R}^d)$, $g\in L^1([0,T]; F^s_{p,r}(\mathbb{R}^d))$. If $f\in L^\infty([0,T]; F^s_{p,r}(\mathbb{R}^d))\bigcap \mathcal{C}([0,T]; \mathcal{S}'(\mathbb{R}^d))$ solves the linear transport equation \eqref{d-tr}, then there exists a constant $C=C(d,p,r,s)$, such that
\begin{align*}
&\|f(t)\|_{F^s_{p,r}(\mathbb{R}^d)}\leq e^{CV(t)}\|f_0\|_{F^s_{p,r}(\mathbb{R}^d)}+\int_0^t\|g(\tau)\|_{F^s_{p,r}(\mathbb{R}^d)}\mathrm{d}\tau \nonumber\\& \quad \quad
+C\int^t_0\left(\|f(\tau)\|_{F^s_{p,r}(\mathbb{R}^d)}\|\nabla v(\tau)\|_{L^\infty(\mathbb{R}^d)}+\|\nabla v(\tau)\|_{F^{s-1}_{p,r}(\mathbb{R}^d)}\|\nabla f(\tau)| |_{L^\infty(\mathbb{R}^d)}\right)\mathrm{d}\tau,
\end{align*}
or
\begin{align*}
&\|f(t)\|_{F^s_{p,r}(\mathbb{R}^d)}\leq e^{CV(t)}\|f_0\|_{F^s_{p,r}(\mathbb{R}^d)}+\int_0^t\|g(\tau)\|_{F^s_{p,r}(\mathbb{R}^d)}\mathrm{d}\tau \nonumber\\& \quad \quad
+C\int^t_0\left(\|f(\tau)\|_{F^s_{p,r}(\mathbb{R}^d)}\|\nabla v(\tau)\|_{L^\infty(\mathbb{R}^d)}+\|\nabla v(\tau)\|_{F^{s}_{p,r}(\mathbb{R}^d)}\|f(\tau)| |_{L^\infty(\mathbb{R}^d)}\right)\mathrm{d}\tau,
\end{align*}
here $V(t)=\int^t_0\|\pa_xv\|_{L^\infty}\dd \tau$. Moreover, if $s>1+\frac{d}{p}$, we have
\bbal
\|f(t)\|_{{F}^s_{p,r}}\leq \left(\|f_0\|_{{F}^s_{p,r}}+\int^t_0e^{-C\widetilde{V}(\tau)}\|g(\tau)\|_{F^s_{p,r}}\dd \tau\right)e^{C\widetilde{V}(t)},
\end{align*}
here $\widetilde{V}(t)=\int^t_0\|v\|_{F^s_{p,r}}\dd \tau$.
\end{lemma}

\section{Local well-posedness and no-uniform dependence}

This section is devoted to establishing the local well-posedness and no-uniform dependence of solutions for the Cauchy problem \eqref{E-P1} in Triebel-Lizorkin spaces. We divide the proof of Theorem \ref{the1} into six steps.

\textbf{Step 1. Construction of the approximate system}

Starting from $u^{0}\triangleq0$, we define by induction a sequence $\{u^{n}\}_{n\in\mathbb{N}}$ of smooth functions   by solving the following linear system:
    \begin{align}\label{lin}
       \begin{cases}
        \partial_tu^{n+1}+u^n\cd \na u^{n+1}=Q(u^n,u^n)+R(u^n,u^n),\\
        u^{n+1}(0,x)=S_{n+1}u_0(x).
        \end{cases}
    \end{align}
Since $u_0\in F^{s}_{p,r}$, then all initial data $S_{n+1}u_0\in F^\infty_{p,r}$. By induction, for every $n\geq 1$ and $T>0$, we can deduce from Lemma \ref{le-tr-t} that the system \eqref{lin} has a unique solution $u^{n}$ in $\mathcal{C}^1([0,T];F^\infty_{p,r})$.

\textbf{Step 2. Uniform estimates to the approximate solutions}

By Lemma \eqref{em-infi}-\ref{le-pro-es1} and \eqref{n-Puv}, we get
\begin{equation}
||Q(u^n,u^n)||_{F^{s}_{p,r}}+||R(u^n,u^n)||_{F^{s}_{p,r}}\leq C||u^n||^2_{F^{s}_{p,r}}.
\end{equation}
Applying Lemma \ref{le-tr-t} to the system \eqref{lin}, we have the following inequality
   \begin{align}\label{l}
   \|u^{n+1}(t)\|_{F^{s}_{p,r}}&\leq e^{CU^{n}(t)}\left(\|S_{n+1}u_0\|_{F^{s}_{p,r}}+C\int^t_0 e^{-CU^{n}(\tau)}
   \|Q(u^{n},u^n),R(u^{n},u^n)\|_{F^{s}_{p,r}} \mathrm{d}\tau\right)\nonumber\\
   &\leq Ce^{CU^{n}(t)} \|u_0\|_{F^{s}_{p,r}}+C\int^t_0 e^{CU^{n}(t)-CU^{n}(\tau)}
   \|u^{n}(\tau)\|^2_{F^{s}_{p,r}}\mathrm{d}\tau,
   \end{align}
   where  $U^{n}(t)\triangleq \int_0^t \|u^{n}(\tau)\|^2_{F^{s}_{p,r}}\mathrm{d}\tau$ and $C\geq1$. We now fix a $T>0$ such that $2C^2\|u_0\|_{F^{s}_{p,r}}T<1$. By induction, we gain
\begin{align*}\label{3.2}
\|u^n(t)\|_{F^{s}_{p,r}}
\leq \frac{C\|u_0\|_{F^{s}_{p,r}}}{{1-2C^2\|u_0\|^2_{F^{s}_{p,r}}t}},~~~\forall t\in[0,T],~~~\forall n\in \mathbb{N}.
\end{align*}
In fact, this is obvious for $n=0$ since $u^0 = 0$. Suppose that it is valid for $n\geq1$, we shall prove it holds for $n+1$.
Direct computations gives that
\begin{align*}
CU^{n}(t)-CU^{n}(t')=C\int^{t}_{t'}\|u^{n}(\tau)\|_{F^{s}_{p,r}}\mathrm{d}\tau
\leq \int^{t}_{t'}\frac{C^2\|u_0\|^2_{F^{s}_{p,r}}}{1-2C^2\|u_0\|^2_{F^{s}_{p,r}}\tau}\mathrm{d}\tau
=\fr12\ln\left(\frac{1-2C^2\|u_0\|^2_{F^{s}_{p,r}}t'}{1-2C^2\|u_0\|^2_{F^{s}_{p,r}}t}\right),
\end{align*}
then from \eqref{l}, we deduce that
\begin{align*}
\|u^{n+1}(t)\|_{F^{s}_{p,r}}\nonumber &\leq C\left(1-2C^2\|u_0\|^2_{F^{s}_{p,r}}t\right)^{-\fr12}\|u_0\|_{_{F^s_{p,r}}}\left(1
+\int^t_0\left(1-2C^2\|u_0\|^2_{F^{s}_{p,r}}t'\right)^{-\fr32}C^2\|u_0\|^{2}_{F^s_{p,r}}\dd t'\right)
\\\nonumber&= \frac{C\|u_0\|_{F^{s}_{p,r}}}{{1-2C^2\|u_0\|^2_{F^{s}_{p,r}}t}}\leq \frac{C\|u_0\|_{F^{s}_{p,r}}}{{1-2C^2\|u_0\|^2_{F^{s}_{p,r}}T}}.
\end{align*}
Therefore, $\{u^{n}\}_{n\in\mathbb{N}}$ is bounded in $L^\infty([0,T];F^{s}_{p,r})$. This entails that $u^{n}\cd \na u^{n+1}$ is bounded  in  $L^\infty([0,T];F^{s-1}_{p,r}).$  As the right-hand side of \eqref{l} is bounded in $L^\infty([0,T];F^{s}_{p,r}),$  we can conclude that $\{u^{n}\}_{n\in\mathbb{N}}$ is bounded in $L^\infty([0,T];F^{s}_{p,r})$.

\textbf{Step 3. Convergence of the approximate solutions}

Note that $F^s_{p,r}\hookrightarrow B^s_{p,\infty}$. Now we claim that $\{u^{n}\}_{n\in\mathbb{N}}$ is a Cauchy sequence in $\mathcal{C}([0,T];B^{s-1}_{p,\infty})$. Indeed, for all $n\in\mathbb{N},$  we have
 \begin{align*}
    \begin{cases}
    \left(\partial_t+(u^{n+m})\cd\na\right)(u^{n+m+1}-u^{n+1})=-(u^{m+n}-u^n)\cd\na u^{n+1} +P(u^{n+m},u^{n+m})-{P}(u^n,u^n),\\
    (u^{n+m+1}-u^{n+1})|_{t=0}=(S_{n+m+1}-S_{n+1})u_0(x).
    \end{cases}
 \end{align*}
By the definition of $S_n$, we get
\begin{equation*}
   (S_{n+m+1}-S_{n+1})u_0(x)=\sum\limits_{q=n+1}^{n+m} \Delta_q u_0(x),
\end{equation*}
which leads to
\begin{align*}
\left\|\sum\limits_{q=n+1}^{n+m} \Delta_q u_0\right\|_{B^{s-1}_{p,\infty}}&=\sup_{j\geq-1}2^{j(s-1)}\left\|\Delta_j\sum\limits_{q=n+1}^{n+m} \Delta_q u_0\right\|_{L^p}\\
&\leq C\sum\limits_{q=n+1}^{n+m} 2^{-q}2^{qs}\left\|\Delta_q u_0\right\|_{L^p}\leq C2^{-n}\|u_0\|_{B^{s}_{p,\infty}}\leq C2^{-n}\|u_0\|_{F^{s}_{p,r}}.
\end{align*}
By Lemmas \ref{em-infi}-\ref{le-pro-es2}, we obtain
\begin{align}
&||P(u^{n+m},u^{n+m})-P(u^n,u^n)||_{B^{s-1}_{p,\infty}}\leq C||u^{n+m}-u^n||_{B^{s-1}_{p,\infty}} (||u^{n+m}||_{B^s_{p,\infty}}+
||u^{n}||_{B^{s}_{p,\infty}}), \label{zhong1}
\\& ||(u^{n+m}-u^n)\cdot \nabla u^{n+1}||_{B^{s-1}_{p,\infty}}\leq C||u^{n+m}-u^n||_{B^{s-1}_{p,\infty}}||u^{n+1}||_{B^s_{p,\infty}}. \label{zhong2}
\end{align}
According to Lemma \ref{le-tr-b} and \eqref{zhong1}-\eqref{zhong2}, we obtain for $t\in[0,T]$
\begin{align*}
  \|(u^{n+m+1}-u^{n+1})(t)\|_{B^{s-1}_{p,\infty}}
  & \leq C
  e^{CU^{n+m}(t)}\Big(\|(S_{n+m+1}-S_{n+1})u_0\|_{B^{s-1}_{p,\infty}}\\
  &\quad+\int^t_0e^{-CU^{n+m}(\tau)}\|(u^{n+m}-u^{n})(\tau)\|_{B^{s-1}_{p,\infty}}
  \|(u^n,u^{n+1},u^{n+m},u^{n+m+1})(\tau)\|_{B^s_{p,\infty}}\dd \tau\Big).
\end{align*}
Define
\begin{equation*}
   b_n^m(t)\triangleq \|(u^{n+m}-u^n)(t)\|_{B^{s-1}_{p,\infty}}.
 \end{equation*}
So we can find a positive $C_T$ independent of $n,m$ such that
\begin{align*}
\begin{cases}
b^m_{n+1}(t)\leq C_{T}\left(2^{-n}+\int_0^t b^m_n(\tau)\mathrm{d}\tau\right),\quad \forall \, t\in[0,T],\\
b^m_1(t)\leq C_T(1+t),\quad \forall \, t\in[0,T].
\end{cases}
\end{align*}
Arguing by induction with respect to the index $n$, we deduce that
\begin{align*}
b^m_{n+1}(t)
&\leq\left(C_{T}\sum\limits_{k=0}^n \frac{(2T C_T)^k}{k!}\right)2^{-n}+ C_T\frac{(T C_T)^{n+1}}{(n+1)!}\rightarrow 0 ,\quad as\ n\rightarrow \infty.
\end{align*}
Now we conclude that $\{u^{n}\}_{n\in\mathbb{N}}$ is a Cauchy sequence in $\mathcal{C}([0,T];B^{s-1}_{p,\infty})$ and converges to a limit function $u\in \mathcal{C}([0,T];B^{s-1}_{p,\infty})$.

\textbf{Step 4. Local existence of a solution}

We will show that $u$ belongs to $\mathcal{C}([0,T];F^{s}_{p,r})$ and satisfies Eqs.\eqref{E-P1}. Since  $\{u^{n}\}_{n\in \mathbb{N}}$ is bounded in $L^\infty(0,T;F^{s}_{p,r})$, Fatou's lemma guarantees that $u$ also belongs to $L^\infty(0,T;F^{s}_{p,r})$. Now, as $\{u^{n}\}_{n\in\mathbb{N}}$
converges to $u \in \ \mathcal{C}([0,T];B^{s-1}_{p,\infty})\hookrightarrow\ \mathcal{C}([0,T];F^{s-1-\eta}_{p,r})$ with $\eta\in(0,s-1)$, an interpolation argument ensures that the convergence actually holds true in $  \mathcal{C}([0,T];F^{s'}_{p,r})$ for any $s' < s.$  It is then easy to pass to the limit in \eqref{lin} and to conclude that
$u$ is a solution to the system \eqref{E-P1}. Furthermore, we will show that the solution $u$ is in $\mathcal{C}([0,T];F^{s}_{p,r})$. In fact, for any $t_1,t_2\in [0,T]$, it follows from the following inequality \eqref{snu0-u0} that
\bbal
\|\mathbf{S}_{t_2}(u_0)-\mathbf{S}_{t_1}(u_0)\|_{F^s_{p,r}}
&\leq \|\mathbf{S}_{t_1}(S_Nu_0)-\mathbf{S}_{t_1}(u_0)\|_{F^s_{p,r}}+\|\mathbf{S}_{t_1}(S_Nu_0)-\mathbf{S}_{t_2}(S_Nu_0)\|_{F^s_{p,r}}\\
&\quad+\|\mathbf{S}_{t_2}(S_Nu_0)-\mathbf{S}_{t_2}(u_0)\|_{F^s_{p,r}}\\
&\leq \sum_{i=1}^2\|\mathbf{S}_{t_i}(S_Nu_0)-\mathbf{S}_{t_i}(u_0)\|_{F^s_{p,r}}+\int_{t_1}^{t_2}\|\pa_{\tau}\mathbf{S}_{\tau}(S_Nu_0)\|_{F^s_{p,r}}\dd\tau\\
&\leq C\|S_Nu_0-u_0\|_{F^s_{p,r}}+C2^N|t_1-t_2|.
\end{align*}
Combining  the density of $S_Nu_0$ in $F^s_{p,r}$ with $r<\infty$  yields that $u\in\mathcal{C}([0,T];F^{s}_{p,r})$.

\textbf{Step 5. Uniqueness of a solution}

In order to prove the uniqueness, we consider two solutions $u$ and $v$ of Eqs.\eqref{E-P1} in
$E^{s}_{p,r}(T)$ with the same initial data. Applying Lemma \ref{lemnew-B-T} below, we assert that $u\equiv v$ in $\mathcal{C}([0,T];F^{s}_{p,r})$. In fact,
$$\|u-v\|_{B^{s-1}_{p,\infty}}\leq C\|u_1(0)-u_2(0)\|_{B^{s-1}_{p,\infty}}=0,$$
which implies the uniqueness.

\begin{lemma}\label{lemnew-B-T}
Let $u\in \mathcal{C}([0,T],F^{s}_{p,r})$ and  $v\in \mathcal{C}([0,T],F^{s+1}_{p,r})$  be two solutions of Eqs.\eqref{E-P1} associated with $u_0$ and $v_0$, respectively. Then we have the estimate for the difference $w=u-v$
\begin{align}
&\|w\|_{B^{s-1}_{p,\infty}}\leq C\|w_0\|_{B^{s-1}_{p,\infty}},\label{y1}
\\&\|w\|_{F^{s}_{p,r}}\leq C\left(\|w_0\|_{F^{s}_{p,r}}+\int^t_0\|\pa_{x}v\|_{F^{s}_{p,r}}\|w\|_{L^\infty}\dd \tau\right),\label{y2}
\end{align}
where $w_0=u_0-v_0$, the constants $C$ depends on $T$ and initial norm $\|u_0,v_0\|_{F^{s}_{p,r}}$.
\end{lemma}
\begin{proof}\;
Obviously, $w\in \mathcal{C}([0,T],F^{s}_{p,r})$ and $w$ solves the transport equation
\begin{align}\label{m}
\begin{cases}
\pa_tw+u\cd\na w=-w\cd\na v+[{P}(u,u)-{P}(v,v)], \\
w(0,x)=u_0(x)-v_0(x).
\end{cases}
\end{align}
Applying Lemma \ref{le-tr-t} to \eqref{m} and using Lemmas  \ref{em-infi}-\ref{le-pro-es2} yields \eqref{y2}. It is easy to obtain \eqref{y1} and we omit the proof.
\end{proof}

\textbf{Step 5. Continuous Dependence}

Letting $u=\mathbf{S}_{t}(u_0)$ and $v=\mathbf{S}_{t}(S_Nu_0)$, using Lemma \ref{le-tr-t}, we have
\bbal
&\|\mathbf{S}_{t}(S_Nu_0)\|_{F^{s+1}_{p,r}}\leq C\|S_Nu_0\|_{F^{s+1}_{p,r}}\leq C2^N\|u_0\|_{F^{s}_{p,r}}
\end{align*}
and
\bbal
\|\mathbf{S}_{t}(S_Nu_0)-\mathbf{S}_{t}(u_0)\|_{L^\infty}&\leq C\|\mathbf{S}_{t}(S_Nu_0)-\mathbf{S}_{t}(u_0)\|_{B^{s-1}_{p,\infty}}
\\&\leq C\|S_Nu_0-u_0\|_{B^{s-1}_{p,\infty}}\\
&\leq C\|S_Nu_0-u_0\|_{F^{s-1}_{p,r}}\\
&\leq C2^{-N}\|S_{N}u_0-u_0\|_{F^{s}_{p,r}},
\end{align*}
which implies
\bal\label{snu0-u0}
\|\mathbf{S}_{t}(S_Nu_0)-\mathbf{S}_{t}(u_0)\|_{F^s_{p,r}}&\leq C\|S_Nu_0-u_0\|_{F^s_{p,r}}.
\end{align}
Then, using the triangle inequality and \eqref{snu0-u0}, we have for $u_0,\widetilde{u}_0\in F^s_{p,r}$,
\bbal
\|\mathbf{S}_{t}(u_0)-\mathbf{S}_{t}(\widetilde{u}_0)\|_{F^s_{p,r}}
&\leq \|\mathbf{S}_{t}(S_Nu_0)-\mathbf{S}_{t}(u_0)\|_{F^s_{p,r}}+\|\mathbf{S}_{t}(S_N\widetilde{u}_0)-\mathbf{S}_{t}(\widetilde{u}_0)\|_{F^s_{p,r}}\\
&~~+\|\mathbf{S}_{t}(S_Nu_0)-\mathbf{S}_{t}(S_N\widetilde{u}_0)\|_{F^s_{p,r}}
\\&\leq C\|S_Nu_0-u_0\|_{F^s_{p,r}}+C\|S_N\widetilde{u}_0-\widetilde{u}_0\|_{F^s_{p,r}}+C\|\mathbf{S}_{t}(S_Nu_0)-\mathbf{S}_{t}(S_N\widetilde{u}_0)\|_{F^s_{p,r}}
\\
&:=\mathbf{I}_1+\mathbf{I}_2+\mathbf{I}_3.
\end{align*}
By the interpolation inequality, one has
\bbal
\mathbf{I}_3&\leq C\|\mathbf{S}_{t}(S_Nu_0)-\mathbf{S}_{t}(S_N\widetilde{u}_0)\|_{B^s_{p,1}}\\&\leq C\|\mathbf{S}_{t}(S_Nu_0)-\mathbf{S}_{t}(S_N\widetilde{u}_0)\|^{\fr12}_{B^{s-1}_{p,\infty}}
\|\mathbf{S}_{t}(S_Nu_0)-\mathbf{S}_{t}(S_N\widetilde{u}_0)\|^{\fr12}_{B^{s+1}_{p,\infty}}\\
&\leq C\|S_Nu_0-S_N\widetilde{u}_0\|^{\fr12}_{B^{s-1}_{p,\infty}}
\|\mathbf{S}_{t}(S_Nu_0)-\mathbf{S}_{t}(S_N\widetilde{u}_0)\|^{\fr12}_{B^{s+1}_{p,\infty}}\leq C2^{\frac N2}\|u_0-\widetilde{u}_0\|^{\frac12}_{F^{s-1}_{p,r}},
\end{align*}
which clearly implies
\bbal
\|\mathbf{S}_{t}(u_0)-\mathbf{S}_{t}
(\widetilde{u}_0)\|_{F^s_{p,r}}
&\lesssim \|S_Nu_0-u_0\|_{F^s_{p,r}}+
\|S_N\widetilde{u}_0-\widetilde{u}_0\|_{F^s_{p,r}}+2^{\frac N2}\|u_0-\widetilde{u}_0\|^{\frac12}_{F^{s-1}_{p,r}}
\\&\lesssim \|S_Nu_0-u_0\|_{F^s_{p,r}}+
\|u_0-\widetilde{u}_0\|_{F^s_{p,r}}+2^{\frac N2}\|u_0-\widetilde{u}_0\|^{\frac12}_{F^{s}_{p,r}}.
\end{align*}
Due to $1<p,r<\infty$, $\forall \varepsilon>0$, one can select $N$ to be sufficiently large, such that
$$C\|S_Nu_0-u_0\|_{F^s_{p,r}}\leq \frac{\varepsilon}{3}.$$
Then fix $N$, we choose $\delta$ so small such that $\|u_0-\widetilde{u}_0\|_{F^{s}_{p,r}}<\delta$ satisfying  $C2^{\frac N2}\delta^{\fr12}<\frac{\varepsilon}{3}$ and $C\delta<\frac{\varepsilon}{3}$. Hence,
\bbal
\|\mathbf{S}_{t}(u_0)-\mathbf{S}_{t}
(\widetilde{u}_0)\|_{F^s_{p,r}}<\varepsilon.
\end{align*}
This completes the proof of continuous dependence.

\textbf{Step 6. Non-uniform continuous dependence}

\begin{proposition}\label{pro 3.1}
Assume that $s>\max\{1+\frac{1}{p}, \frac{3}{2}\}$ with $1< p, r< \infty$ and $\|u_0\|_{F^s_{p,r}}\lesssim 1$. Then under the assumptions of Theorem \ref{the1}, we have
\bbal
\|\mathbf{S}_{t}(u_0)-u_0+t\mathbf{U}_0\|_{F^{s}_{p,r}}\lesssim t^{2}\left[1+\|u_0\|_{F^{s-1}_{p,r}}\Big(\|u_0\|_{F^{s+1}_{p,r}}+\|u_0\|_{F^{s-1}_{p,r}}\|u_0\|_{F^{s+2}_{p,r}}\Big)\right],
\end{align*}
where $\mathbf{U}_0=u_0\cd\na u_0-Q(u_0,u_0)-R(u_0,u_0)$.
\end{proposition}
\begin{proof}\; For simplicity, denote $u(t)=\mathbf{S}_t(u_0)$. Firstly, by the local well-posedness result, there exists a positive time $T=T(\|u_0\|_{F^s_{p,r}},s,p,r)$ such that the solution $u(t)$ belongs to $\mathcal{C}([0, T];  F_{p, r}^s)$. Moreover,  for all $t\in[0,T]$ and $\gamma\geq s$,  we have
\bal\label{u-estimate}
\|u(t)\|_{B^{s-1}_{p,\infty}}\leq C\|u_0\|_{B^{s-1}_{p,\infty}},\quad \|u(t)\|_{F^\gamma_{p,r}}\leq C\|u_0\|_{F^\gamma_{p,r}}.
\end{align}

Now we shall estimate the Triebel-Lizorkin norm and Besov norm of the term $u(t)-u_0$, separately. By the Mean Value Theorem, Remark \ref{re-bt} and Lemmas \ref{em-infi}-\ref{le-pro-es1}, we obtain
\bal\label{zyl1}
\|u(t)-u_0\|_{F^s_{p,r}}
&\leq \int^t_0\|\pa_\tau u\|_{F^s_{p,r}} \dd\tau\nonumber\\
&\leq \int^t_0\|Q(u,u)\|_{F^s_{p,r}} \dd\tau+\int^t_0\|R(u,u)\|_{F^s_{p,r}} \dd\tau+ \int^t_0\|u \cd\na u\|_{F^s_{p,r}} \dd\tau\nonumber\\
&\lesssim t\left(\|u\|^2_{L_t^\infty(F^{s}_{p,r})}
+\|u\|_{L_t^\infty(L^\infty)}\|\na u\|_{L_t^\infty(F^{s}_{p,r})}+\|\na u\|_{L_t^\infty(L^\infty)}\|u\|_{L_t^\infty(F^{s}_{p,r})}\right)\nonumber\\
&\lesssim t\left(\|u_0\|^2_{F^{s}_{p,r}}+\|u_0\|_{B^{s-1}_{p,\infty}}\|u_0\|_{F^{s+1}_{p,r}}\right)\nonumber\\
&\lesssim t\left(1+\|u_0\|_{F^{s-1}_{p,r}}\|u_0\|_{F^{s+1}_{p,r}}\right),
\end{align}
where the third inequality follows from the fact that $F_{p, r}^{s-1}$ is a Banach algebra with $s-1>\max\{\frac{1}{p}, \frac{1}{2}\}$.
Taking the same argument as above, by Remark \ref{re-bt} and Lemmas \ref{em-infi}-\ref{le-pro-es2}, one can deduce that
\bal\label{zyl2}
\|u(t)-u_0\|_{B^{s-1}_{p,\infty}}
&\leq \int^t_0\|\pa_\tau u\|_{B^{s-1}_{p,\infty}} \dd\tau
\nonumber\\&\les \int^t_0\|Q(u,u)\|_{B^{s-1}_{p,\infty}} \dd\tau+\int^t_0\|R(u,u)\|_{B^{s-1}_{p,\infty}} \dd\tau+ \int^t_0\|u\cd\na u\|_{B^{s-1}_{p,\infty}} \dd\tau\nonumber\\&\leq \int^t_0\|\na u\|_{B^{s-2}_{p,\infty}}\|\na u\|_{B^{s-1}_{p,\infty}} \dd\tau+\int^t_0\|\na u\|_{B^{s-2}_{p,\infty}}\| u\|_{B^{s-1}_{p,\infty}}  \dd\tau+ \int^t_0\|u\|_{B^{s-1}_{p,\infty}}\|u\|_{B^{s}_{p,\infty}} \dd\tau
\nonumber\\&\lesssim t\|u_0\|_{F^{s-1}_{p,r}}
\end{align}
and
\bal\label{zyl3}
\|u(t)-u_0\|_{F^{s+1}_{p,r}}
&\leq \int^t_0\|\pa_\tau u\|_{F^{s+1}_{p,r}} \dd\tau
\nonumber\\&\leq \int^t_0\|Q(u,u)\|_{F^{s+1}_{p,r}} \dd\tau+\int^t_0\|R(u,u)\|_{F^{s+1}_{p,r}} \dd\tau+ \int^t_0\|u\cd\na u\|_{F^{s+1}_{p,r}} \dd\tau\nonumber\\
&\lesssim t\left(\|\na u,u\|_{L_t^\infty(L^\infty)}\|\na u\cdot u\|_{L_t^\infty(F^{s}_{p,r})}
+\|u\|_{L_t^\infty(L^\infty)}\|\na u\|_{L_t^\infty(F^{s+1}_{p,r})}+\|\na u\|_{L_t^\infty(L^\infty)}\|u\|_{L_t^\infty(F^{s}_{p,r})}\right)
\nonumber\\&\lesssim t\left(\|u_0\|_{F^{s+1}_{p,r}}+\|u_0\|_{F^{s-1}_{p,r}}\|u_0\|_{F^{s+2}_{p,r}}\right).
\end{align}

Next, we estimate the Triebel-Lizorkin norm for the term $u(t)-u_0+t\mathbf{U}_0$. Applying the Mean Value Theorem again, it follows from the estimates \eqref{zyl1}-\eqref{zyl3} that
\bbal
\|u(t)-u_0+t\mathbf{U}_0\|_{F^s_{p,r}}
&\leq \int^t_0\|\partial_\tau u+\mathbf{U}_0\|_{F^s_{p,r}} \dd\tau \\
&\leq \int^t_0\|Q(u,u)-Q(u_0,u_0)\|_{F^s_{p,r}} \dd\tau+\int^t_0\|R(u,u)-R(u_0,u_0)\|_{F^s_{p,r}} \dd\tau \\
&\quad+\int^t_0\|u\cd\na u-u_0\cd\na u_0\|_{F^s_{p,r}}\dd\tau\\
&\lesssim
t^{2}\left[1+\|u_0\|_{F^{s-1}_{p,r}}\Big(\|u_0\|_{F^{s+1}_{p,r}}+\|u_0\|_{F^{s-1}_{p,r}}\|u_0\|_{F^{s+2}_{p,r}}\Big)\right],
\end{align*}
where we have used
\begin{align*}
\|Q(u,u)-Q(u_0,u_0)\|_{F^s_{p,r}} +\|R(u,u)-R(u_0,u_0)\|_{F^s_{p,r}}&\les\|u-u_0\|_{F^s_{p,r}}
\end{align*}
and
\begin{align*}
\|u\cd\na u-u_0\cd\na u_0\|_{F_{p, r}^s}&\les\|u^2-u^2_0\|_{F_{p, r}^{s+1}}\les\|(u-u_0)(u+u_0)\|_{F_{p, r}^{s+1}}
\\
&\les\|u-u_0\|_{L^\infty}\|u_0,u\|_{F_{p, r}^{s+1}}+\|u-u_0\|_{F_{p, r}^{s+1}}\|u_0,u\|_{L^\infty}\\
&\les\|u-u_0\|_{B_{p, \infty}^{s-1}}\|u_0\|_{F_{p, r}^{s+1}}+\|u-u_0\|_{F_{p, r}^{s+1}}\|u_0\|_{F_{p, r}^{s-1}}.
   \end{align*}
Thus, we finish the proof of Proposition \ref{pro 3.1}.
\end{proof}

Let $\hat{\phi}$ be in $C_0(\mathbb{R})$ such that
\begin{equation*}
\hat{\phi}(\xi)=
\begin{cases}
1, \quad |\xi|\leq \frac{1}{4^d},\\
0, \quad |\xi|\geq \frac{1}{2^d}.
\end{cases}
\end{equation*}
First, we choose velocities $u^n_0$ and $v^n_0$ with the following forms:
\bbal
u^n_0=
\Big(f_n,0,\cdots,0\Big), \qquad v^n_0=
\Big(f_n+g_n,0,\cdots,0\Big),
\end{align*}
with
\bbal
&f_n(x)=
2^{-ns}\phi(x_1)\sin(\frac{17}{12}2^nx_1)\phi(x_2)\cdots \phi(x_d),
\qquad n \in \mathbb{Z},
\\&g_n(x)=
2^{-n}\phi(x_1)\phi(x_2)\cdots \phi(x_d),
\qquad n \in \mathbb{Z}.
\end{align*}
An easy computation gives that
\bbal
\hat{f}_n=2^{-ns-1}i\big[\hat{\phi}
\big(\xi_1+\frac{17}{12}2^n\big)-\hat{\phi}\big(\xi_1-\frac{17}{12}2^n\big)\big]
\hat{\phi}(\xi_2)\cdots \hat{\phi}(\xi_d),
\end{align*}
which implies
\bbal
\mathrm{supp} \ \hat{f}_n\subset \Big\{\xi\in\R^d: \ \frac{17}{12}2^n-\fr12\leq |\xi|\leq \frac{17}{12}2^n+\fr12\Big\}.
\end{align*}
Then, we deduce that
\bbal
\Delta_j(f_n)=
\begin{cases}
f_n, \quad  j=n,\\
0, \qquad j\neq n.
\end{cases}
\end{align*}
According to Definition \ref{de-t}, one can show that for $k\geq -1$,
\bbal
||u^n_0||_{F^{s+k}_{p,r}}\leq C2^{kn}, \quad ||v^n_0||_{F^{s+k}_{p,r}}\leq C2^{kn}, \quad ||v^n_0,\na v^n_0||_{L^\infty}\leq C(2^{-n}+2^{-(s-1)n}).
\end{align*}
Let $\mathbf{S}_{t}(u^n_0)$ and $\mathbf{S}_{t}(v^n_0)$ be the solution of \eqref{E-P1} with initial data $u^n_0$ and $v^n_0$, respectively. Obviously, we have
\bal\label{es-gn}
\|u^n_0-v^n_0\|_{F^s_{p,r}}=\|g_n\|_{F^s_{p,r}}\leq C2^{-n},
\end{align}
and thus
\bbal
\lim_{n\to\infty}\|u^n_0-v^n_0\|_{F^s_{p,r}}=0.
\end{align*}
It is easy to show that for $\sigma\geq s$,
\bbal
&\|u^n_0\|_{F^{s-1}_{p,r}}\lesssim 2^{-n},\quad\|v^n_0\|_{F^{s-1}_{p,r}}\lesssim 2^{-n},\qquad \|u^n_0,v^n_0\|_{F^{\sigma}_{p,r}}\leq C2^{(\sigma-s)n}.
\end{align*}
Therefore,
\bbal
&\|v_0^n\|_{F^{s-1}_{p,r}}\Big(\|v_0^n\|_{F^{s+1}_{p,r}}+\|v_0^n\|_{F^{s-1}_{p,r}}\|v_0^n\|_{F^{s+2}_{p,r}}\Big)\lesssim  1,
\\&\|u_0^n\|_{F^{s-1}_{p,r}}\Big(\|u_0^n\|_{F^{s+1}_{p,r}}+\|u_0^n\|_{F^{s-1}_{p,r}}\|u_0^n\|_{F^{s+2}_{p,r}}\Big)\lesssim  1.
\end{align*}
From Remark \ref{re-bt} and Lemmas \ref{em-infi}-\ref{le-pro-es1}, it follows that
\bal\label{zhong-R}
\|Q(u^n_0,u^n_0)-Q(v^n_0,v^n_0)\|_{F^s_{p,r}}+\|R(u^n_0,u^n_0)-R(v^n_0,v^n_0)\|_{F^s_{p,r}}\leq C||u^n_0-v^n_0||_{F^s_{p,r}}(||u^n_0||_{F^s_{p,r}}+||v^n_0||_{F^s_{p,r}}).
\end{align}
By Proposition \ref{pro 3.1}, we deduce from \eqref{es-gn}-\eqref{zhong-R} that
\bal\label{yyh}
&\quad \ \|\mathbf{S}_{t}(u^n_0)-\mathbf{S}_{t}(v^n_0)\|_{F^s_{p,r}}
\nonumber \\ \geq& \|u^n_0-tu^n_0\cd\na u^n_0-tQ(u^n_0,u^n_0)-tR(u^n_0,u^n_0)-\Big(v^n_0-tv^n_0\cd\na v^n_0-tQ(v^n_0,v^n_0)-tR(v^n_0,v^n_0)\Big)||_{F^s_{p,r}}-Ct^2
\nonumber\\ \geq &~t\left\|u^n_{0}\cd \na u^n_{0}-v^n_{0}\cd \na v^n_{0}\right\|_{F^s_{p,r}}-\|u^n_0-v^n_0\|_{F^s_{p,r}}
\nonumber \\& \qquad -t\left\|Q(u^n_0,u^n_0)-Q(v^n_0,v^n_0)\right\|_{F^s_{p,r}}-t\left\|R(u^n_0,u^n_0)-R(v^n_0,v^n_0)\right\|_{F^s_{p,r}}-Ct^2\nonumber\\
\geq&~ t\left\|u^n_{0}\cd \na u^n_{0}-v^n_{0}\cd \na v^n_{0}\right\|_{F^s_{p,r}}-C2^{-n}-Ct^{2}.
\end{align}
As $\big(v^n_0\cd\na v^n_0-u^n_0\cd\na u^n_0\big)_i=0$ for $i=2,\cdots d$, and
\begin{align}\label{397}
\big(v^n_0\cd\na v^n_0-u^n_0\cd\na u^n_0\big)_1=f_n\cd\pa_{x_1} g_n+g_n\cd\pa_{x_1} f_n+g_n\cd\pa_{x_1} g_n.
\end{align}
By the definitions of $f_n$ and $g_n,$ we have
\bbal
&\quad \ ||g_n\cd\pa_{x_1} g_n,f_n\cd\pa_{x_1} g_n||_{F^s_{p,r}}
\\&\leq C||f_n||_{L^\infty}||g_n||_{F^{s+1}_{p,r}}+C||\na g_n||_{L^\infty}||f_n||_{F^s_{p,r}}+C||g_n||^2_{F^{s+1}_{p,r}}
\leq C2^{-\frac12n},
\end{align*}
which along with \eqref{yyh} and \eqref{397} implies
\bal\label{1000}
\|\mathbf{S}_{t}(u^n_0)-\mathbf{S}_{t}(v^n_0)\|_{F^s_{p,r}}\geq ct||g_n\cd\pa_{x_1} f_n||_{F^s_{p,r}}-Ct^2-C2^{-\frac12n}.
\end{align}
Note that
\bbal
\mathrm{supp} \ \mathcal{F}({g_n\cd\pa_{x_1} f_n})\subset \Big\{\xi\in\R^d: \ \frac{17}{12}2^n-1\leq |\xi|\leq \frac{17}{12}2^n+1\Big\}.
\end{align*}
Thus,
\bbal
\Delta_j(g_n\cd\pa_{x_1} f_n)=
\begin{cases}
g_n\cd\pa_{x_1} f_n, \quad  j=n,\\
0, \qquad j\neq n.
\end{cases}
\end{align*}
Consequently,
\bal\label{1001}
||g_n\cd\pa_{x_1} f_n||_{F^s_{p,r}}&=2^{ns}||g_n\cd\pa_{x_1} f_n||_{L^p}\nonumber
\\&\geq ||\phi^2(x_1)\cos(\frac{17}{12}2^nx_1)\phi^2(x_2)\cdots \phi^2(x_d)||_{L^p} \nonumber \\& \quad -2^{-n}||\phi(x_1)\phi'(x_1)\sin(\frac{17}{12}2^nx_1)\phi^2(x_2)\cdots \phi^2(x_d)||_{L^p} \nonumber
\\&\geq \frac{17}{12}||\phi^2(x_1)\cos(\frac{17}{12}2^nx_1)||_{L^p}||\phi||^{2(d-1)}_{L^{2p}}-C2^{-n}.
\end{align}
Using Riemann-Lebesgue's lemma, we get
\begin{eqnarray*}
      \liminf_{n\rightarrow \infty} \|g_n\pa_{x_1}f_n\|_{F^s_{p,r}}\geq c,
        \end{eqnarray*}
which along with \eqref{1000} yields
\bbal
\liminf_{n\rightarrow \infty}\|\mathbf{S}_{t}(u^n_0)-\mathbf{S}_{t}(v^n_0)\|_{F^s_{p,r}}\gtrsim t\quad\text{for} \ t \ \text{small enough}.
\end{align*}
This completes the proof of non-uniform continuous dependence.

\section*{Acknowledgements}
M. Li is supported by the Jiangxi Provincial Natural Science Foundation (20232BAB201013). J.
Li is supported by the National Natural Science Foundation of China (12161004), Training Program
for Academic and Technical Leaders of Major Disciplines in Ganpo Juncai Support Program
(20232BCJ23009) and Jiangxi Provincial Natural Science Foundation (20224BAB201008).

\section*{Declarations}
\noindent\textbf{Data Availability} No data was used for the research described in the article.

\vspace*{1em}
\noindent\textbf{Conflict of interest}
The authors declare that they have no conflict of interest.

\end{document}